\documentclass[12pt]{amsart}
\usepackage{amssymb,latexsym}
\theoremstyle{definition}
\newtheorem{theorem}{Theorem}

\newtheorem{proposition}[theorem]{Proposition}
\newtheorem{lemma}[theorem]{Lemma}
\newtheorem{definition}[theorem]{Definition}
\newtheorem{example}[theorem]{Example}

\newtheorem{remark}[theorem]{Remark}

\date{}
\newcommand{\numberset}{\mathbb}

\newcommand{\Z}{\numberset{Z}}

\newcommand{\F}{\numberset{F}}

\newcommand{\Pro}{\numberset{P}}

\newcommand{\Ol}{\mathcal{O}}

\usepackage{hyperref}

\usepackage[margin=2.5cm]{geometry}

\usepackage{mathptmx}

\usepackage{amsaddr}

\begin{document}

\title[]{The dual geometry of Hermitian two-point codes}

\author{Edoardo Ballico$^1$}
\address{Department of Mathematics, University of Trento\\Via Sommarive 14,
38123 Povo (TN), Italy}
\email{$^1$edoardo.ballico@unitn.it}

\author{Alberto Ravagnani$^2$$^*$}
\address{Institut de Math\'ematiques, Universit\'e de Neuch\^{a}tel\\
Emile-Argand 11, CH-2000 Neuch\^{a}tel, Switzerland}
\email{$^2$alberto.ravagnani@unine.ch}

\thanks{$^1$Partially supported by MIUR and GNSAGA}
\thanks{$^*$Corresponding author}
\subjclass[2010]{94B27; 14C20; 11G20}
\keywords{Hermitian code; Goppa code; two-point code;
minimum-weight codeword.}

\maketitle

\providecommand{\bysame}{\leavevmode\hbox to3em{\hrulefill}\thinspace}

\begin{abstract}
 In this paper we study the algebraic geometry of any two-point code on the
Hermitian curve and reveal the purely geometric nature of their dual minimum
distance. We describe the minimum-weight codewords of many of their
dual codes through an explicit geometric characterization of their supports. In particular, we show that  they appear as sets of collinear points in many cases. 
\end{abstract}

\section{Introduction} \label{storia}

Goppa codes were introduced by the Russian Mathematician V. D. Goppa in 1970
(see \cite{Goppa}), who had the idea 
of employing algebraic curves to construct error correcting codes.

\begin{definition}\label{Goppacode}
 Let $q$ be a prime power, and let $\Pro^k$ be the projective space of dimension
$k$ over the field
 $\F_q$. Consider a smooth curve $X\subseteq \Pro^k$ defined over $\F_q$ and an $\F_q$-rational
divisor $D$ on $X$. 
Take points $P_1,...,P_n \in X(\F_q)$ not lying in the support of $D$, and set
 $\overline{D}:=\sum_{i=1}^n P_i$.   
The \textbf{Goppa code} $C(\overline{D},D)$ is defined to be the code obtained
evaluating the Riemann-Roch space 
space $L(D)$ at the points $P_1,...,P_n$.
\end{definition}

From the definition we see that curves carrying many rational points may produce
long codes. On the other hand, 
the number of
$\F_q$-rational points of a smooth curve $X$ defined over $\F_q$ is bounded by the well known Hasse-Weil bound:
$$|X(\F_q)|-q-1 \le 2g(X) \sqrt{q},$$
where $g(X)$ is the geometric genus of $X$. As a consequence, curves attaining
the bound (which are called \textbf{maximal}) are particularly 
interesting in coding theory.
The Hermitian curve (see Section \ref{intr} for definition and properties) is a
plane smooth maximal curve defined over finite fields of 
the form $\F_{q^2}$. \textbf{One-point} codes from the Hermitian curve (i.e.,
codes obtained by evaluating vector spaces of the form $L(mP)$, with $m \in \Z$
and $P$ a rational point of the curve) are probably the most studied algebraic
geometric codes (see, among the others, \cite{Sti}, Section 8.3, and
\cite{SMP}).
\textbf{Two-point} codes from the Hermitian curve are obtained evaluating
Riemann-Roch spaces of the form $L(mP+nQ)$, 
with $a,b \in \Z$ and $P,Q$ distinct rational points of the curve.
The minimum distance of any two-point code from the Hermitian curve has been completely determined by Homma and
Kim in \cite{HK1}, \cite{HK2}, and \cite{HK3}. Using different techniques, Park
provides in
\cite{Park} explicit formulas for their dual minimum distances. The results by
Park will be stated later in this 
paper (see Section \ref{intr}).

 Recently, A. Couvreur found a way to characterize the dual minimum distance of
some geometric Goppa codes in terms of their projective geometries (see
\cite{c3}). We extended his approach in \cite{br} in order to find out the dual
minimum distances of some $m$-points codes on the Hermitian curve for $m \ge
3$. 
Here we combine the results by Park, the method of \cite{c3}, and the
projective geometry of the Hermitian curve in order to obtain a geometric
description of the minimum-weight codewords of
many duals of Hermitian two-point codes. To be precise, we will show
that the points in their supports obey some particular geometric laws (such as collinearity). 

\begin{remark}
The capability to determine minimum-weight or, more generally, small-weight
codewords of Goppa codes is relevant for 
cryptographic issues. Indeed, in \cite{crypto1} and \cite{crypto2} it is shown
that
minimum-weight and small-weight codewords of Goppa codes can be used to attack
the well known McEliece 
cryptosystem (see \cite{mcel} and the references within it). 
\end{remark}

The remainder of the paper is organized as follows. In Section \ref{intr} we
state Park's results and interpret two-point Hermitian codes in terms of
evaluations of cohomology groups.
The geometric properties of the Hermitian curve are summarized in Section \ref{sec2}, where we also characterize the dual minimum distance of codes from curves in terms of non-vanishing conditions
on cohomology groups. We apply the results of Section \ref{sec2} in Section \ref{geommm}, describing the supports of the minimum-weight codewords of the duals of two-point  codes from the Hermitian curve.
Computational examples can be found in Section \ref{computmagma}.

\section{Preliminaries}
\label{intr}
Let $q$ be a prime power, and let $\Pro^2$ denote the projective plane over the
field $\F_{q^2}$ with coordinates $(x:y:z)$. Let $X \subseteq \Pro^2$ be the Hermitian curve (see
\cite{Sti}, Example VI.3.6) of affine equation $$y^q+y=x^{q+1}.$$ It is well known
that $X$ is a maximal curve with $q^3+1$ $\F_{q^2}$-rational points (see  \cite{Ste} among the others). Choose $m,n \in \Z$ and two distinct points $P,Q \in
X(\F_{q^2})$. The code obtained evaluating the vector space $L(mP+nQ)$ on
$X(\F_{q^2})\setminus \{ P,Q\}$ is said to be a (classical) two-point code on
$X$ (\cite{HK1}, Section 1). Clearly, we may consider only the pairs $(m,n)$
such that $n+m>0$, Otherwise the resulting code has dimension at most one, and it is
not of interest for applications. Since the group of
automorphisms of $X$, say $\mbox{Aut}(X)$, is 2-transitive (\cite{AC}, pages
572-573) we can assume, without loss of generality,  $P=P_\infty$ and
$Q=P_0$, where $P_\infty$ is the  only point at infinity of $X(\F_{q^2})$, of
projective coordinates $(0:1:0)$, and $P_0$ is the point with coordinates $(0:0:1)
\in X(\F_{q^2})$. Set $B:=X(\F_{q^2}) \setminus \{ P_\infty,P_0\}$ and denote by
$C_{m,n}$ the two-point code obtained evaluating the vector space
$L(mP_\infty+nP_0)$ on $B$.

\begin{remark}
An explicit basis of $L(mP_\infty+nP_0)$ is well known
for any $m,n \in \Z$ (see \cite{DK} and \cite{MM2005} for applications in coding theory). 
\end{remark}

We will widely
use the following two results\footnote{If $X$ is the Hermitian curve and $P \in
X(\F_{q^2})$ then any canonical $\F_{q^2}$-rational divisor $K$ on $X$ satisfies the linear equivalence
$K \sim (q-2)(q+1)P$.}.
\begin{theorem}[\cite{Park}, Theorem 3.3]
\label{3.3}
Let $G=K+mP_\infty + nP_0$, where $K$ is a canonical divisor on $X$. Let
$C^\perp$ be the dual of the code obtained evaluating the vector space $L(G)$ on
the points of $B$ and denote by $\delta$ its minimum distance. Write
\begin{eqnarray*}
m&=&m_0(q+1)-m_1, \ \ \  0 \le m_1 \le q,\\
n&=&n_0(q+1)-n_1, \ \ \ \ 0 \le n_1 \le q. 
\end{eqnarray*}
Set $d^*:=m+n$ and suppose that $G$ satisfies either
\begin{itemize}
\item[(a)] $\deg(G)>\deg(K)+q$, or
\item[(b)] $\deg(K) \le \deg(G) \le \deg(K)+q$ and $G \not \sim sP_\infty$ and
$G \not \sim tP_0$, for all $s,t \in \Z$.
\end{itemize}
The following formulas hold.
\begin{enumerate}
\item If $0 \le m_1,n_1 \le m_0+n_0$, then $\delta=d^*$.
\item If $0 \le n_1 \le m_0+n_0 < m_1$, then $\delta=d^*+m_1-(m_0+n_0)$.
\item If $0 \le m_1 \le m_0+n_0 < n_1$, then  $\delta=d^*+n_1-(m_0+n_0)$.
\item If $m_0+n_0 < m_1 \le n_1 < q$, then $\delta=d^*+m_1+n_1-2(m_0+n_0)$.
\item If $m_0+n_0 < n_1 \le m_1 < q$, then $\delta=d^*+m_1+n_1-2(m_0+n_0)$.
\item If $m_0+n_0<m_1,n_1$ and $m_1=n_1=q$, then $\delta=d^*+q-(m_0+n_0)$.
\end{enumerate}
\end{theorem}

\begin{theorem}[\cite{Park}, Theorem 3.5]
\label{3.5}
Let $G=mP_\infty + nP_0$ and let $C^\perp$ be the dual of the code obtained
evaluating the vector space $L(G)$ on the points of $B$. Denote by $\delta$ its
minimum distance and write
\begin{eqnarray*}
m&=&m_0(q+1)+m_1, \ \ \  0 \le m_1 \le q, \\
n&=&n_0(q+1)+n_1, \ \ \ \ 0 \le n_1 \le q. 
\end{eqnarray*}
Suppose that $G$ satisfies either
\begin{itemize}
\item[(a)] $\deg(G)< \deg(K)$, or
\item[(b)] $\deg(K) \le \deg(G) \le \deg(K)+q$ and $G \sim sP_\infty$ or $G \sim
tP_0$, for some $s,t \in \Z$.
\end{itemize}
Then $\delta=m_0+n_0+2$.
\end{theorem}

\begin{remark} \label{rrrrr}
 Codes $C,D \subseteq \F_q^n$ are said to be \textbf{strongly isometric} if
$C=vD$, where
$v \in \F_q^n$ is a vector of non-zero components and 
$$vD:=\{ (v_1d_1,v_2d_2,...,v_nd_n) : (d_1,d_2,...,d_n) \in D \}.$$
 A strong isometry is an equivalence relation of codes. Strongly isometric
codes have the same
minimum distance and the same weight distribution. Two codes are strongly
isometric
if and only if their dual codes are strongly isometric. A strong isometry
preserves the supports 
of the codewords.
\end{remark}

The Hermitian curve $X \subseteq \Pro^2$ has a very particular projective geometry (see for example
 \cite{AC}). Let us briefly recall some basic properties (other
important facts will be stated in Section \ref{sec2}). For any $P,Q \in
X(\F_{q^2})$ we have a linear equivalence $(q+1)P\sim(q+1)Q$. Moreover, for any
$P \in X(\F_{q^2})$ we have an isomorphism of shaves
$\Ol_X(1)\cong\mathcal{L}((q+1)P)$, the latter one being the invertible sheaf
associated to the divisor $(q+1)P$ on $X$. Using these geometric facts we easily
deduce that for any pair of integers $(m,n)$ such that $m+n>0$ there exists a
tern of integers $(d,a,b)$ such that $d>0$, $0 \le a,b \le q$ and
$L(mP_\infty+nP_0) \cong H^0(X,\Ol_X(d)(-E))$, where $E:=aP_\infty+bP_0$ (as a zero-dimensional subscheme of $\Pro^2$). 

\begin{remark} \label{strongiso}
Notice that isometries of codes induced by linear equivalences of divisors are
strong isometries 
(see for example \cite{mp}, Remark 2.16). As a consequence, 
$C_{m,n}$ turns out to be strongly isometric to the two-point code
obtained evaluating
the vector space $H^0(X,\Ol_X(d)(-E))$ on $B$, here denoted by $C(d,a,b)$.
\end{remark}

\begin{remark} \label{remar2}
  By the 2-transitivity of $\mbox{Aut}(X)$, we may also assume $a \le b$ by
permuting (if necessary) $P_0$ and $P_\infty$. Moreover, notice that if $b=0$ then
$C(d,a,b)$ is not a code of interest. Indeed, $C(d,a,0)$ is a one-point code
whose parameters can be easily improved by evaluating $H^0(X,\Ol_X(d)(-E))$ on
$X(\F_{q^2})\setminus \{P_\infty\}$ instead of $B$, obtaining a longer code. Hence, from now on, we will
implicitly assume $a \le b$ and  $b \neq 0$ when writing $C(d,a,b)$.
\end{remark}

\begin{remark}
For $m \in \Z_{>0}$, consider the Hermitian one-point code $C_m$ obtained
evaluating the vector space $L(mP_\infty)$ on $X(\F_{q^2}) \setminus \{ P_\infty
\}$. Hermitian one-point codes can be easily studied employing the geometric method here presented. In fact, the simple structure of the one-point space $L(mP_\infty)$ can be used to characterize 
the small-weight codewords of Hermitian one-point codes and also study possible improvements of such codes (see 
\cite{onep}).
\end{remark}

\section{Coding on the Hermitian curve} \label{sec2}
In this section we state some technical lemmas on the geometry of the Hermitian
curve and certain zero-dimensional subschemes of $\Pro^2$ (see also
\cite{bgi} and \cite{ep}). Then we apply these results to the duals of
Hermitian two-point codes,
characterizing their minimum-weight codewords in terms of vanishing conditions of cohomology groups.

\begin{lemma}
\label{intrette}
Let $X$ be the Hermitian curve. Every line $L$ of $\Pro^2$ either intersects $X$
in $q+1$ distinct $(\F_{q^2})$-rational points, or $L$ is tangent to $X$ at a
point $P$ (with contact order $q+1$). In the latter case $L$ does not intersect
$X$ in any other $\F_{q^2}$-rational point different from $P$.
\end{lemma}
\begin{proof}
See \cite{pj}, part (i) of Lemma 7.3.2, at page 247.
\end{proof}

\begin{lemma}
\label{intcurve}
Let $X\subseteq \Pro^2$ be the Hermitian curve. Fix an integer $e\in \{2,\dots
,q+1\}$ and $P\in 
 X(\mathbb {F}_{q^2})$. Let $E\subseteq X$ be the divisor $eP$, seen as a closed
degree $e$ subscheme of $\mathbb {P}^2$. Let $L_{X,P}\subseteq \mathbb {P}^2$ be
the tangent line
 to $X$ at $P$. Let $T\subseteq \mathbb {P}^2$ be any effective divisor
 (i.e. a plane curve, possibly with multiple components)
 of degree $\le e-1$ and containing $E$. Then $L_{X,P}\subseteq T$, i.e.
$L_{X,P}$ is one of the components of $T$.
\end{lemma}

\begin{proof}
Since $L_{X,P}$ has order of contact $q+1\ge e$ with $X$ at $P$, we have
$E\subseteq L_{X,P}$.
Since we have $\deg (E)>\deg (T)$ and $E\subseteq T\cap L_{X,P}$, Bezout theorem implies
$L_{X,P}\subseteq T$.
\end{proof}

\begin{lemma} \label{tgline}
 Let $X \subseteq \Pro^2$ be the Hermitian curve and let $P \in X(\F_{q^2})$.
Denote by $L_{X,P}$ the tangent line to $X$ at $P$. Let $C \subseteq \Pro^2$ be
any irreducible smooth curve defined over $\F_{q^2}$ such that $P \in
C(\F_{q^2})$. Let $L_{C,P}$ be the tangent line to $C$ at $P$. Then $2P
\subseteq C$ if and only if $L_{X,P}=L_{C,P}$. 
\end{lemma}
\begin{proof}
 Use both the definition of tangent line and Lemma \ref{intcurve}, with
$T:=L_{C,P}$ and $E:=2P$.
\end{proof}

\begin{lemma}\label{u00.01}
Let $X\subseteq \Pro^2$ be the Hermitian curve. Choose an integer $d>0$ and 
a zero-dimensional scheme $Z \subseteq X(\F_{q^2})$ of degree $z>0$. The
following facts hold.

\begin{itemize}
\item[(a)] If $z\le d+1$, then $h^1(\mathbb {P}^2,\mathcal {I}_Z(d))=0$.

\item[(b)] If $d+2 \le z\le 2d+1$, then $h^1(\mathbb {P}^2,\mathcal {I}_Z(d))>0$
if and only if there
exists a line $T_1$ such that $\deg (T_1\cap Z)\ge d+2$.

\item[(c)] If $2d+2\le z \le 3d-1$ and $d\ge 2$, then $h^1(\mathbb
{P}^2,\mathcal {I}_Z(d))>0$ if and only if
either there
exists a line $T_1$ defined over $\F_{q^2}$ such that $\deg (T_1\cap Z)\ge d+2$,
or there exists a conic $T_2$ defined over $\F_{q^2}$ such that $\deg (T_2\cap
Z) \ge 2d+2$.

\item[(d)] Assume $z=3d$ and $d\ge 3$. Then $h^1(\mathbb {P}^2,\mathcal
{I}_Z(d))>0$ if and only if
either there
exists a line $T_1$defined over $\F_{q^2}$ such that $\deg (T_1\cap Z)\ge d+2$,
or there is a conic $T_2$ defined over $\F_{q^2}$ such that $\deg (T_2\cap Z)
\ge 2d+2$, or there exists a plane cubic
$T_3$ such that $Z$ is the complete intersection of $T_3$ and a plane curve of
degree
$d$. In the latter case, if $d\ge 4$ then $T_3$ is unique and defined over
$\F_{q^2}$ and we may
find a plane curve $C_d$ defined over $\F_{q^2}$ and with $Z = T_3\cap C_d$.

\item[(e)] Assume $z \le 4d-5$ and $d\ge 4$.  Then $h^1(\mathbb {P}^2,\mathcal
{I}_Z(d))>0$ if and only if
either there
exists a line $T_1$ defined over $\F_{q^2}$ such that $\deg (T_1\cap Z)\ge d+2$,
or there exists a conic $T_2$ defined over $\F_{q^2}$ such that $\deg (T_2\cap
Z) \ge 2d+2$, or there exist $W\subseteq Z$ defined over $\F_{q^2}$ with $\deg
(W)=3d$ and plane cubic
$T_3$ defined over $\F_{q^2}$ such that $W$ is  the complete intersection of
$T_3$ and a plane curve of degree
$d$, or there is a plane cubic $C_3$ defined over $\F_{q^2}$ such that $\deg
(C_3\cap Z)\ge 3d+1$.
\end{itemize}
\end{lemma}

\begin{proof}
 See \cite{br}, Lemma 7.
\end{proof}

The previous lemma implies the following useful result.
\begin{lemma}\label{u1}
Let $X \subseteq \Pro^2$ be the Hermitian curve. Fix an integer $m>0$ and a zero-dimensional scheme $Z\subseteq X(\F_{q^2})$ such that $2 \le \deg(Z) \le \max\{2m+2, 3m,4m-5\}$. If $h^1(\mathbb
{P}^2,\mathcal {I}_Z(m))>0$, then there exists a subscheme $W\subseteq Z$ such
that one of
the following cases occurs.
\begin{enumerate}
\item[(a)] $\deg (W)=m+2$ and $W$ is contained in a line;
\item[(b)] $\deg (W) = 2m+2$ and $W$ is contained in a conic;
\item[(c)] $\deg (W) =3m$ and $W$ is the complete intersection of a
curve
of degree $3$ and a curve of degree $m$;
\item[(d)] $\deg (W)=3m+1$ and $W$ is contained in a cubic curve.
\end{enumerate}
\end{lemma}
\begin{proof}
 Apply  Lemma \ref{u00.01}.
\end{proof}

\begin{lemma}\label{e1}
Fix any smooth plane curve $X\subseteq\mathbb {P}^2$, an integer $d>0$, a
zero-dimensional
scheme $E\subseteq X$ and a finite
subseteq $B\subseteq X$ such that $B\cap E_{\mbox{red}}=\emptyset$. Let $C$ be the
code on $X$ obtained
evaluating the vector space $H^0(X, \mathcal {O}_X(d)(-E))$ at the points of
$B$. Set
$c:= \deg (X)$.
Assume $\sharp (B) > dc$. We set $n:= \sharp (B)$ and $k:= h^0(X,\mathcal
{O}_X(d))-\deg (E) $, where
$h^0(X,\mathcal {O}_X(d)) =\binom{d+2}{2}$ if $d < c$ and $h^0(X,\mathcal
{O}_X(d))=\binom{d+2}{2}
-\binom{d-c+2}{2}$ if $d \ge c$. Then $C$ is a code of length $n$ and dimension $k$. Moreover, the minimum
distance
of $C^{\bot}$ is the minimal cardinality, say $s$, of a subset $S \subseteq B$
of $B$ such that
$h^1(\mathbb {P}^2,\mathcal {I}_{S\cup E}(d)) >h^1(\mathbb {P}^2,\mathcal
{I}_E(d))$. A given codeword of $C^{\bot}$
has weight $s$ if and only if it is supported by a subset $S\subseteq B$ such that
$\sharp (S) = s$ and $h^1(\mathbb {P}^2,\mathcal {I}_{E\cup S}(d)) >h^1(\mathbb
{P}^2,\mathcal {I}_E(d))$.
\end{lemma}

\begin{proof}
The computation of $h^0(X,\mathcal {O}_X(d))$ is well known. We impose that $B$
does
not intersect the support of $E$. The case $E=\emptyset$ is a particular case of
\cite{c3}, Proposition
3.1. In the general case notice that $C$ is obtained evaluating
a family of homogeneous degree $d$ polynomials (the ones vanishing on the scheme
$E$)
at the points of $B$.
Since $X$ is projectively normal, the restriction map $\rho _d: H^0(\mathbb {P}^2,\mathcal {O}_{\mathbb {P}^2}(d)) \to H^0(X,\mathcal
{O}_X(d))$ is surjective.
As a consequence, the restriction map 
\begin{eqnarray*}
\rho _{d,E}: H^0(\mathbb {P}^2,\mathcal {I}_E(d)) \to H^0(X,\mathcal
{O}_X(d)(-E))
\end{eqnarray*} is surjective. Hence a finite subset $S\subseteq X\setminus
E_{\mbox{red}}$
imposes independent condition to the space $H^0(X,\mathcal {O}_X(d)(-E))$
if and only
if $S$ imposes independent conditions to $H^0(\mathbb {P}^2,\mathcal {I}_E(d))$.
On the other hand, $S$ imposes independent conditions to $H^0(\mathbb
{P}^2,\mathcal {I}_E(d))$
if and only if $h^1(\mathbb {P}^2,\mathcal {I}_{E\cup S}(d)) =h^1(\mathbb
{P}^2,\mathcal {I}_E(d))$. Notice that here we use again that $S\cap E=\emptyset$.
\end{proof}

\begin{remark} \label{remar}
 In the notation of Lemma \ref{e1}, a given minimum-weight codeword of
$C^\perp$, whose support is $S \subseteq B$, trivially satisfies the condition
$h^1(\Pro^2,\mathcal{I}_{E \cup S}(d))>0$, which is geometrically described by Lemma \ref{u1}.
\end{remark}

\begin{proposition} \label{pr1}
 Fix
integers $d>2$, $a_i \in \{0,...,q\}$ for $i=1,2$ with $a_1 \le a_2$. Let $r$ be
(if it exists) the maximum integer $i\le 2$
such
that $a_i\le d-2+i$. Otherwise, simply set $r:=0$. Define $d':= d-2+r>0$. If $r>0$ set $a'_i:= a_i$ for $i\le r$ and $a'_i:=0$
otherwise. If $r=0$ set $a'_i:=0$ for any $i \in \{1,2\}$. Let $X\subset
\mathbb
{P}^2$ be the Hermitian curve and let $P_1:=P_\infty$, $P_2:=P_0$.  Define $E:=
\sum_{i=1}^2 a_i P_i$. If $r>0$ define $E':= \sum _{i=1}^{2}a'_iP_i$. Otherwise
set $E':=0$. Then the codes obtained evaluating the two vector spaces
$H^0(X,\Ol_X(d)(-E))$ and $H^0(X,\Ol_X(d')(-E'))$ on $B=X(\F_{q^2})\setminus
\{P_0,P_\infty \}$ are strongly isometric (Remark \ref{rrrrr}).
\end{proposition}

\begin{proof}
 If $r=2$ then $E=E'$, $d=d'$ and so we have nothing to prove. If $r=1$ then
$d'=d-1$, $a_1'=a_1$, $a_2'=0$ and $E':=a_1P_1$. 
 Consider $E$ and $E'$ as closed subschemes of $\Pro^2$ of degree $a_1+a_2$ and
$a_1$, respectively. Since the restriction map 
$H^0(\mathbb {P}^2,\mathcal {O}_{\mathbb {P}^2}(d))
\to H^0(X,\mathcal {O}_X(d))$
is surjective, then the restriction maps
\begin{equation*}
\rho _E: H^0(\mathbb {P}^2,\mathcal {I}_{E}(d)) \to
H^0(X,\mathcal
{O}_X(d)(-E)), \ \ \ \ \ 
 \rho _{E'}:H^0(\mathbb {P}^2,\mathcal {I}_{E'}(d))
\to
H^0(X,\mathcal{O}_X(d)(-E'))
\end{equation*} are surjective themselves.
Every tangent line $T_PX$ to $X$ at a point $P\in X(\mathbb {F}_q)$
has order of contact $q+1$ with $X$ at $P$ (Lemma \ref{intrette}) and hence, by
Bezout's
theorem,
it intersects $X$ only
at $P$. As a consequence $T_{P_2}X\cap E$ has degree $a_2$. Since $a_2>d$ we get
that
every degree $d$
homogeneous form vanishing on $E$ vanishes also on the line $T_{P_2}X$,
i.e., it
is divided by the equation of $T_{P_2}X$. Since $\rho _E$ and
$\rho _{E'}$ are surjective and $B\cap T_{P_2}X =\emptyset$,
we deduce that
the codes obtained 
evaluating $H^0(X,\mathcal {O}_X(d)(-E))$ and $H^0(X,\mathcal
{O}_{X}(d')(-E'))$, respectively, are in fact strongly isometric. In the case $r=0$,
repeat the same procedure of case $r=1$ with $a_2$ and $d$, and then with $a_1$
and $d-1$.
\end{proof}

Combining Lemma \ref{u1}, Remark \ref{remar} and Proposition \ref{pr1}, we get
the following result.
\begin{proposition}\label{pr2}
Fix integers $d>2$, $a_i \in \{0,...,q\}$ for $i=1,2$. Let $r$, $d'$, $E$ and
$E'$ be as in Proposition \ref{pr1}. Denote by $C:=C(d,a_1,a_2)$ the code
obtained evaluating the space $H^0(X,\Ol_X(d)(-E))$  on
$B=X(\F_{q^2})\setminus \{P_0,P_\infty \}$ and let $C^\perp$ be its dual code.
Let $S \subseteq B$ be the support of a minimum-weight codeword of $C^\perp$ and
assume $\deg(E')+\sharp(S) \le \max\{2d'+2,3d',4d'-5\}$. There exists a
zero-dimensional scheme $W \subseteq E' \cup S$ which satisfies one of the
following properties.
\begin{enumerate}
\item[(a)] $\deg (W)=d'+2$ and $W$ is contained in a line;
\item[(b)] $\deg (W) = 2d'+2$ and $W$ is contained in a conic;
\item[(c)] $\deg (W) =3d'$ and $W$ is the complete intersection of a
curve
of degree $3$ and a curve of degree $d'$;
\item[(d)] $\deg (W)=3d'+1$ and $W$ is contained in a cubic curve.
\end{enumerate}
\end{proposition}

In the following section we apply the previous results to explicitly characterize the supports of the minimum-weight codewords of the duals of Hermitian two-point codes. 

\section{Geometry of minimum-weight codewords} \label{geommm}

Following Park's Theorems \ref{3.3} and \ref{3.5}, we consider the following three groups,
 (G1), (G2) and (G3), of two-point Hermitian codes $C(d,a,b)$ (we always assume
 $0 \le a \le b \le q$ and $b \neq 0$ as in Remark \ref{remar2}). We will denote
by $K$ any canonical ($\F_{q^2}$-rational) divisor on $X$.
\begin{enumerate}
\item[(G1)] The codes $C(d,a,b)$ such that $0<d(q-1)-a-b<\deg(K)$.
\label{gruppi}
\item[(G2)] The codes $C(d,a,b)$ such that $\deg(K) \le d(q-1)-a-b \le
\deg(K)+q$ and $d(q+1)P_\infty-aP_\infty-bP_0$ is not linearly equivalent to a
multiple of $P_\infty$ nor to a multiple of $P_0$.
\item[(G3)] The codes $C(d,a,b)$ such that $\deg(K) \le d(q-1)-a-b \le
\deg(K)+q$ and $d(q+1)P_\infty-aP_\infty-bP_0$ is linearly equivalent to a
multiple of $P_\infty$ or to a multiple of $P_0$ .
\end{enumerate}

\begin{theorem} \label{1gruppo}
Let $C:=C(d,a,b)$ be a code of group (G1), with $d>2$, $0 \le a \le b \le q$, and
$b \neq 0$. Set $a_1:=a$, $a_2:=b$  and define $r$, $d'$, $E$, $E'$ as in
Proposition \ref{pr1}.
If $a \ne 0$, then the minimum distance of $C^\perp$ is $d$. Let
$\{P_1,...,P_d\}$ be the support of a minimum-weight codeword.
\begin{enumerate}
 \item[(a)] If $r=0$, then $P_1,...,P_d$ are collinear.
\item[(b)] If $r=1$, then $P_\infty,P_1,...,P_d$ are collinear.
\item[(c)] If $r=2$ and $d>3$, then $P_1,...,P_d$ are collinear.
\item[(d)] If $r=2$ and $d=3$, then
\begin{enumerate}
\item[(d.1)] either $P_1,...,P_d$ are collinear,
\item[(d.2)] or they lie on the union of the line through $P_0$ and $P_\infty$
and a smooth conic which is tangent to $X$ at both $P_0$ and $P_\infty$,
\item[(d.3)] or they lie on an irreducible cubic which is tangent to $X$ at both
$P_0$ and $P_\infty$.
\end{enumerate}
\end{enumerate}
If $a=0$, then the minimum distance of $C^\perp$ is $d+1$ and $r \neq 0$. Let
$\{P_1,...,P_{d+1}\}$ be the support of a minimum-weight codeword of $C^\perp$.
\begin{enumerate}
\item[(a)] If $r=1$, then $P_1,...,P_d$ are collinear.
\item[(b)] If $r=2$, then $P_0,P_1,...,P_d$ are collinear.
\end{enumerate}
\end{theorem}

\begin{proof}
Denote by $L_{X,P_0}$, $L_{X,P_\infty}$ the tangent lines to $X$ at $P_0$ and
$P_\infty$ (respectively). Denote by $L_{0,\infty}$ the line through $P_0$ and
$P_\infty$. Let us consider first the case $a \ne 0$. Keep in mind the linear
equivalence
\begin{equation*}
 d(q+1)P_\infty-aP_\infty-bP_0 \sim ((d-2)(q+1)+(q+1-a))P_\infty+(q+1-b)P_0
\end{equation*}
and apply Theorem \ref{3.5} with $m_0:=d-2$, $m_1:=q+1-a$, $n_0:=0$ and
$n_1:=q+1-b$, obtaining that the minimum distance of $C^\perp$ is $m_0+n_0+2=d$
(here we use $b \neq 0$). Let $S:=\{ P_1,...,P_d\} \subseteq B$ be the support
of a minimum-weight codeword of $C^\perp$. If $r=0$, then $d'=d-2$. Since
$\deg(E')+\sharp(S)=d'+2$, Proposition \ref{pr2} applies: there exists a
subscheme $W \subseteq \left\{P_1,...,P_d \right\}$ of degree $d'+2=d-2+2=d$
contained in a line, and we are done. If $r=1$, then $a \leq d-1 <b$ and
$d':=d-2+1=d-1$. Since $\deg(E')+\sharp(S) \le (d-1)+d=2d-1=2(d'+1)-1=2d'+1$,
Proposition \ref{pr2}  applies and there exists a subscheme $W \subseteq
aP_\infty \cup \{P_1,...,P_d \}$ of degree $d+1$ and contained in a line $L$. It
is immediately seen that $P_\infty \in W$ and $L \neq L_{X,P_\infty}$  (since $a
\le d-1$ and Lemma \ref{intrette} applies). Hence the points $P_\infty,
P_1,...,P_d$ are collinear.
If $r=2$ then $d'=d$ and $a \leq b \leq d$. In this case $E'=E=aP_\infty+bP_0$.
We have $\deg(E)+\sharp(S)\le 3d$ and then Proposition \ref{pr2} applies: either
there exists a subscheme $W \subseteq aP_\infty+bP_0 \cup \{ P_1,...,P_d \}$ of
degree $d+2$ and contained in a line $L$, or there exists a subscheme $W
\subseteq aP_\infty+bP_0 \cup \{ P_1,...,P_d \}$ of degree $2d+2$ and contained
in a conic $T$, or there exists a subscheme  $W \subseteq aP_\infty+bP_0 \cup \{
P_1,...,P_d \}$ of degree $3d$  which is the complete intersection of a curve of
degree 3 and one of degree $d$.
In the former case, since $\deg (W)=d+2$, we see that both $P_\infty$ and $P_0$
must lie on $L$, unless $L=L_{X,P_\infty}$ or $L=L_{X,P_0}$. Since these cases are ruled out
by Lemma \ref{intrette}, all the points $P_1,...,P_d$ lie on the line
$L_{0,\infty}$.
Now assume $\deg(W)=2d+2$ and $W$ contained in a conic $T$. The sum of the
multiplicities of $P_\infty$ and $P_0$ in $W$, say $e_W(P_\infty)$ and
$e_W(P_0)$, must satisfy $e_W(P_\infty)+e_W(P_0) \ge (2d+2)-d=d+2>4$. Since
$e_W(P_\infty) \le d$ and $e_W(P_0) \le d$, we deduce (Lemma \ref{intcurve})
$T=L_{X,P_0} \cup L{X,P_0}$, which contradicts Lemma \ref{intrette}. 
Finally, assume $\deg(W)=3d$. Then $a=b=d>2$ and $dP_\infty+dP_0+P_1,...,P_d$ is
the complete intersection of a cubic curve $C$ and a curve of degree $d$. If
$d>3$, then Lemma \ref{intcurve} implies $L_{X,P_\infty} \subseteq C$ and
$L_{X,P_0} \subseteq C$. By Lemma \ref{intrette}, the points $P_1,...,P_d$ are collinear
 (they lie on the line $C-L_{X,P_\infty} -L_{X,P_0}$). If $d=3$, then it
occurs exactly one of the following cases.
\begin{itemize}
 \item The cubic $C$ is reducible and a degree one component of $C$, say $L$, is
tangent to $X$.
\begin{itemize}
\item If $L \neq L_{X,P_\infty}$ and $L \neq L_{X,P_0}$, then Lemma
\ref{intrette} and Lemma \ref{intcurve} imply (keeping in mind our assumption
$d>2$) $L_{X,P_0}, L_{X,P_\infty} \subseteq C - L$. This is forbidden by Lemma
\ref{intrette}. 

\item  If $L=L_{X,P_\infty}$ (or $L=L_{X,P_0}$) then ($d>2$) Lemma
\ref{intcurve} and Lemma \ref{intrette} imply also $L_{X,P_0} \subseteq C-
L_{X,P_\infty}$ (or $L_{X,P_\infty} \subseteq C- L_{X,P_0}$). Hence
$P_1,...,P_d$ lie on the line $C -L_{X,P_\infty} - L_{X,P_0}$. 
\end{itemize}
\item The cubic $C$ is reducible and no degree one component of it is tangent to
$X$. Let $L\subseteq C$ be a line. By hypothesis, $L$ is not tangent to $X$.
\begin{itemize}
\item If $L$ is not the line through $P_0$ and $P_\infty$, then $3P_\infty
\subseteq C- L$ or $3P_0 \subseteq C - L$ (or both). In any case, since
$2P_\infty+2P_0 \subseteq C - L$, we have $L_{X,P_\infty} \subseteq C$ and
$L_{X,P_0} \subseteq C$ (here we used Lemma \ref{intrette} again). As a
consequence, $P_1,...,P_d$ lie on $L$ (Lemma \ref{intrette}).
\item If $L$ is the line through $P_0$ and $P_\infty$, then either $C=L\cup
L_{X,P_0} \cup L_{X,P_\infty}$ and $P_1,...,P_d$ lie on $L$ (by Lemma
\ref{intrette}), or $C=L \cup T$, and $T$ is a smooth conic tangent to $X$ at
both $P_0$ and $P_\infty$ (remember that $2P_\infty+2P_0 \subseteq C - L$ and
Lemma \ref{tgline} holds).
\end{itemize}
\item The cubic $C$ is an irreducible curve and tangent (by Lemma \ref{tgline})
to $X$ at both $P_0$ and $P_\infty$.
\end{itemize}
Now consider the simpler case $a=0$. Theorem \ref{3.5} proves that the minimum
distance is $d+1$.
If $r=0$, then $b>d$ and $0=a>d-1$, but we assumed $d>2$.
If $r=1$, then  $\deg(E') +\sharp(S) =\sharp(S)=d+1=d'+2$ and Proposition
\ref{pr2} applies: there exists subscheme $W \subseteq \{ P_1,...,P_{d+1}\}$ of
degree $d'+2=d+1$ and contained in a line. In particular, the points $P_1,...,P_{d+1}$ turn
out to be collinear.
If $r=2$, then $b \le d$, $E'=E$ and $d'=d$. We see that $\deg(E)+\sharp(S) \le
d+(d+1)=2d+1$ and Proposition \ref{pr2} applies: there exists a subscheme $W
\subseteq bP_0 \cup \{P_1,...,P_d \}$ of degree $d+2$ and contained in a line
$L$. By Lemma \ref{intrette}, $L$ cannot be the tangent line to $X$ at $P_0$. As
a consequence, $P_0$ appears is $L$ with multiplicity exactly one and so
$P_0,P_1,...,P_{d+1}$ are collinear. 
\end{proof}

To complete our analysis we need the following geometric result.
\begin{lemma} \label{gonalit}
 Let $X$ be the Hermitian curve. Let $P,Q \in X(\F_{q^2})$ with $P \neq Q$.
Assume $nP \sim nQ$ for a certain $0 \le n \le q$. Then $n=0$.
\end{lemma}

\begin{proof}
 Assume $n \neq 0$. Let $h:=\mbox{gcd}(n,q+1) \le n \le q$. Write $h=\alpha
n+\beta(q+1)$. We deduce the linear equivalences $hP \sim \alpha n
P+\beta(q+1)P$ and $hQ \sim \alpha n Q+\beta(q+1)Q$. Since $nP \sim nQ$ then
$\alpha n P \sim \alpha n Q$. Hence $hP -\beta(q+1)P \sim hQ -\beta(q+1)Q$.
Since $\beta(q+1)P \sim \beta(q+1)Q$ (as described in Section \ref{intr}) we
have $hP \sim hQ$ for a certain $1 \le h \le q$ which divides $q+1$ (in particular, we have $h<q$). 
Notice that $hQ$ is an effective divisor linear equivalent to $hP$. It follows that 
the linear system $|hP|$ contains $hQ$. Hence $\dim_{\F_q}L(hP) \ge 2$. Since $0 \le h < q$, from the explicit computation of the Weierstrass semigroup
at $P$ (see \cite{anote}, Proposition 1) we deduce $\dim_{\F_q}L(hP) = 1$, a
contradiction.
\end{proof}

\begin{remark} \label{no3}
 If $C(d,a,b)$ is in group (G3), then we have either $d(q+1)-aP_\infty-bP_0 \sim
s P_\infty$ for a certain $s \in \Z$,
or $d(q+1)-aP_\infty-bP_0 \sim t P_0$ for a certain $t \in \Z$. In the former
case $s=d(q+1)-a-b$, and hence $bP_0 \sim bP_\infty$. Since $b \le q$, we deduce
(see Lemma \ref{gonalit}) $b=0$, which contradicts our assumptions (Remark
\ref{remar2}). In the latter case we have $d(q+1)P_\infty-aP_\infty-bP_0 \sim
d(q+1)P_0-aP_0-bP_0$.
Since $(q+1)P_\infty \sim (q+1)P_0$ (see Section \ref{intr}), we deduce
$aP_\infty \sim aP_0$ and then (Lemma \ref{gonalit} again) $a=0$.
\end{remark}

Since we assumed $b \neq 0$ (Remark \ref{remar2}) and a canonical divisor $K$ on
$X$ has degree\footnote{Any smooth plane curve of degree $c$ has genus
$g(X)=(c-1)(c-2)/2$. The degree of a canonical divisor $K$ on $X$ is always
$2g(X)-2$. Remember that the Hermitian curve has degree $q+1$.} $q^2-q-2$, we
immediately see that if $d<q-1$ then $C(d,a,b)$ is of group (G1), for any choice
of $0 \le a \le b \le q$. If $d=q-1$ and $a \neq 0$ then $C(d,a,b)$ could be
either in group (G1), or in group (G2). Indeed, $(q-1)(q+1)-a-b \le q^2-2$ for
any $1 \le a \le b \le q$ and group (G3) is excluded by Remark \ref{no3}.
On the other hand, $C(q-1,0,b)$ is always in group (G3), because $q^2-q-2 \le
(q-1)(q+1)-b \le q^2-2$ for any $1 \le b \le q$. This proves that in order to
complete our study of case $d=q-1$ we can analyse separately the codes
$C(q-1,a,b)$ of group (G2), which satisfy $1 \le a \le b \le q$, and the codes
of the form $C(q-1,0,b)$ (obviously $b \le q$), which are in group (G3). Note
that in case $d=q-1$ our usual assumption $d>2$ is equivalent to $q > 3$.

\subsection{Codes $C(q-1,a,b)$ of group (G2)}
In the following Remark \ref{car2} we give a characterization of the codes
$C(q-1,a,b)$ of group (G2).
\begin{remark}\label{car2}
 A code $C(q-1,a,b)$ of group (G2) satisfies (combining Remark \ref{remar2} and
Remark \ref{no3}) $1 \le a \le b \le q$. The condition of being in group (G2) is
\begin{equation*}
 q^2-q-2 \le (q-1)(q+1)-a-b \le q^2-2,
\end{equation*}
which is equivalent to $a+b \le q+1$ (the right-hand side condition always
holds). 
\end{remark}
Now consider the linear equivalence
\begin{eqnarray*}
(q-1)(q+1)P_\infty-aP_\infty-bP_0 &\sim& (q-2)(q+1)P_\infty
+(q+1-a)P_\infty-bP_0 \\
&\sim& K + (q+1-a)P_\infty-bP_0
\end{eqnarray*}
and apply Theorem \ref{3.3} with $m=q+1-a$, $n=-b$, $m_0=1$, $m_1=a$, $n_0=0$
and $n_1=b$. Since we assumed $a \le b$ and Remark \ref{car2} holds, only Park's
cases (1), (3) and (4) of Theorem \ref{3.3} are of our interest.  The dual
minimum distance, say $\delta(q-1,a,b)$, of a code $C(q-1,a,b)$ of group (G2)
can be read in Table \ref{md2}.

\begin{table}[h!]
\centering
\small{
\begin{center}
\begin{tabular}[h!]{|p{2.6cm}|p{2.4cm}|p{2.4cm}|p{2.4cm}|}
\hline
Park's case $\rightarrow$ & (1) & (3) & (4)  \\
\hline
\hline
$\delta(q-1,a,b)$ & $q+1-a-b$ & $q-a$ & $q-1$  \\
\hline
\end{tabular}
\end{center}
}
\caption{Dual minimum distance of a code $C(q-1,a,b)$ of group (G2).}
\label{md2}
\end{table}

\begin{theorem} \label{te2}
 Let $C(q-1,a,b)$ be a code of group (G2) (conditions $d=q-1>2$, $1 \le a \le b
\le q$ and $a+b \le q+2$ will be implicitly assumed). Let
$\delta:=\delta(q-1,a,b)$ be the dual minimum distance of $C(q-1,a,b)$, and let
$S:=\{ P_1,...,P_\delta\}$ be the support of a minimum-weight codeword. Then one of the following cases occurs.
\begin{enumerate}
 \item[(a)] $P_1,...,P_\delta$ lie on the line through $P_0$ and $P_\infty$.
\item[(b)] $P_\infty, P_1,...,P_\delta$ are collinear.
\item[(c)] $P_0, P_1,...,P_\delta$ are collinear.
\end{enumerate}
\end{theorem}
\begin{proof}
Let us examine separately cases (1), (3) and (4) of Theorem \ref{3.5}. In the
notation of Proposition \ref{pr1}, set $a_1:=a$, $a_2:=b$, $d:=q-1$,
$E:=aP_\infty+bP_0$, and define $E'$ and $d'$ as in the statement.
\begin{enumerate}
 \item If $C(q-1,a,b)$ is described by case (1) of Theorem \ref{3.5}, then $a,b
\le 1$. Hence $a=b=1$ and  the  dual minimum distance is $\delta=q-1$ (see Table
\ref{md2}). Since $d>2$, we have $E'=E$ and $d'=d$. Notice that
$\deg(E')+\sharp(S)=q-1+2=q+1=d+2$ and Proposition \ref{pr2} applies: there
exists a subscheme $W \subseteq P_\infty+P_0 \cup \{ P_1,...,P_{q-1}\}$ of
degree $q+1$ and contained in a line. This means that $P_1,...,P_{q-1}$ lie on
the line through $P_\infty$ and $P_0$.
\item If $C(q-1,a,b)$ is described by case (3) of Theorem \ref{3.5}, then $0 \le
a \le 1 < b$. Hence $a=1$. 
\begin{enumerate}
\item[(2.i)] If $b<q$, then $d'=d$, $E'=E$ and the dual minimum distance is
$\delta=q-a=q-1$. Notice that $\deg(E')+\sharp(S)=(q-1)+(1+b)=q+b=d+1+b$, and
Proposition \ref{pr2} applies: there exists a subscheme $W \subseteq
P_\infty+bP_0 \cup \{ P_1,...,P_{q-1}\}$ of degree $q+1$ and contained in a
line\footnote{The case $\deg(W)\ge 2d+2=2q$ implies $b \ge q$, which contradicts
our hypothesis on $b$.}. By Lemma \ref{intrette}, $P_0$ must appear in $W$ with
multiplicitly exactly one and $P_1,...,P_{q-1}$ lie on the line through
$P_\infty$ and $P_0$.
\item[(2.ii)] If $b=q=d+1$, then $d'=d-1$, $E'=P_\infty$ and the dual minimum
distance is $\delta=q-a=q-1$. Since $\deg(E')+\sharp(S)=q-1+1=q=d'+2$,
Proposition \ref{pr2} applies and there exists a subscheme $W \subseteq P_\infty
\cup \{P_1,...,P_{q-1}\}$ of degree $d'+2=q$ and contained in a line. In other
words, the points $P_\infty,P_1,...,P_d$ turn out to be collinear.
\end{enumerate}
\item If $C(q-1,a,b)$ is described by case (4) of Theorem \ref{3.5}, then $1 < a
\le b <q$. In particular, $b \le d$, $d'=d$ and $E'=E$. The dual minimum distance
is $\delta=q-1$ and $\deg(E)+\sharp(S)=q-1+a+b=d+a+b \le d+(q+1)=2d+2$. Hence
Proposition \ref{pr2} applies: either there exists a subscheme $W \subseteq
aP_\infty+bP_0 \cup \{P_1,...,P_{q-1} \}$ of degree $q+1$ and contained in a
line $L$, or there exists a subscheme $W \subseteq aP_\infty+bP_0 \cup
\{P_1,...,P_{q-1}\}$ of degree $2q$ and contained in a conic $T$. In the former
case, Lemma \ref{intrette} implies that $P_\infty$ and $P_0$ appear in $W$ with
multiplicity exactly one, and so $P_1,...,P_{q-1}$ lie on the line through
$P_\infty$ and $P_0$. In the latter case it is immediately seen that
$W=aP_\infty+bP_0+\sum_{i=1}^{q-1}P_i$. Since $\deg(W)=2q$, we have $a+b=q+1$.
Since our general assumption $d=q-1>2$ holds, it follows $\min \{a,b \}>2$ and
so either $L_{X,P_\infty} \subseteq T$, or $L_{X,P_0} \subseteq T$ (Lemma
\ref{intcurve}). In any case, by Lemma \ref{intrette}, either $P_0,P_1,...,P_{q-1}$
are collinear, or $P_\infty,P_1,...,P_{q-1}$ are collinear.
\end{enumerate}
This concludes our proof.
\end{proof}

\subsection{Codes $C(q-1,0,b)$ of group (G3)}
If $C(q-1,0,b)$ is a code of group (G3) (i.e. $d=q-1>2$,  $a=0$, $1 \le b \le
q$), then its minimum distance, $\delta(q-1,0,b)$, is given by Theorem \ref{3.5}.
Indeed, choose $m=(q-1)(q+1)$, $m_0=q-1$, $m_1=0$, $n=-b$, $n_0=-1$, $n_1=q+1-b$,
and compute $\delta(q-1,0,b)=q$. Let $S:= \{P_1,...,P_q \}$ be the support of a
minimum-weight codeword. In the notation of Proposition \ref{pr1}, set $a_1:=0$,
$a_2:=b$, $E:=bP_0$ and define $d'$ and $E'$ as in the statement of cited proposition. If $b \le
d=q-1$, then $E'=E$ and $\deg(E)+\sharp(S)\le q-1+q=2q-1=2d+1$. Hence Proposition
\ref{pr2} applies and gives a zero-dimensional scheme $W \subseteq bP_0 \cup \{
P_1,...,P_q\}$ of degree $q+1$ and contained in a line. By Lemma \ref{intrette},
the multiplicity of $P_0$ in $W$ must be exactly one, and then $P_0,P_1,...,P_q$
lie on $L$.
If $d=q$, then $E'=0$ and $d'=d-1=q-2$. Since $\deg(E')+\sharp(S)=q=d+1$,
Proposition \ref{pr2} applies: there exists a $W \subseteq \{ P_1,...,P_q\}$ of
degree $d'+2=(d-1)+2=q$ and contained in a line. It follows that the points $P_1,...,P_q$
are collinear.

Let us summarize the previous analysis in the following concise result.
\begin{theorem} \label{te3}
 Consider a code $C(q-1,0,b)$ of group (G3). Its dual minimum distance is $q$.
Let $\{ P_1,...,P_q\}$ be the support of a minimum-weight codeword. Then
$P_1,...,P_q$ are collinear. Moreover, if $b \le q-1$, then 
$P_0,P_1,...,P_q$ are collinear.
\end{theorem}

\section{Computational Examples} \label{computmagma}
When $q$ is  small, our results can be checked by writing simple
\texttt{MAGMA} programs  (see the homepage \url{http://magma.maths.usyd.edu.au}). 

\begin{example} \label{exxx}
 Let $q=4$ ($q^2=16$) and let $X$ be the Hermitian curve defined over $\F_{16}$
by the affine equation $y^4+y=x^5$. Consider the Hermitian two-point code $C$
obtained evaluating the vector space $L(10P_\infty+3P_0)$ on
$X(\F_{16})\setminus \{P_\infty,P_0\}$ (here $P_\infty$ and $P_0$ are those
defined in Section \ref{intr}). Write $10P_\infty+3P_0 \sim
d(q+1)P_\infty-aP_\infty-bP_0$ with $(d,a,b)=(3,0,2)$. Note that $0 \le a \le b
\le q$ and $d=q-1$. Hence our code $C$ is in fact the code $C(3,0,1)$, with
$q=4$, described by Theorem \ref{te3}. If $\F_{16}= \langle \alpha \rangle$ and
$\alpha^4+\alpha+1=0$ then a (random) minimum-weight codeword of $C^\perp$ is
\begin{eqnarray*}
 &(& 0, 0, 0, 0, 0, 0, 0, 0, 0, 0, 0, 0, 0, 0, 0, 0, 0, 0, 0, 0, 0, 0, 0, 0, 0,
0, 0, 0, 0, 1, 0, 0, \\ & & 0, 0, 0, 0, 0, 0, 0, 0, 0, 0, 0, 0, 0, \alpha^3, 0,
0, 0, 0, 0, 0, 0, 0, 0, \alpha^{12}, 0, 0, 0, \alpha^9, 0, 0, 0 \ \ )
\end{eqnarray*}
with support made of the points of affine coordinates
\begin{equation*}
 P_1:=(\alpha^7, \alpha^7),\ \ P_2:=(\alpha^{11}, \alpha^{11}),\ \
P_3:=(\alpha^{13}, \alpha^{13}),\ \ P_4:=(\alpha^{14}, \alpha^{14}).
\end{equation*}
Notice that the points are in fact four (Table \ref{md2} says that the dual
minimum distance of $C$ has to be 4) and they lie, together with $P_0$, on the
line defined over $\F_{16}$ by the homogeneous equation $x-y=0$, as described in
Theorem \ref{te3}.
\end{example}

\begin{example}
 In the same notations of Example \ref{exxx}, choose $q=5$ and $(d,a,b)=(4,1,1)$.
Notice that $d=q-1$ and $b=1$. Hence $C(4,1,1)$ (with $q=5$) is described by
Theorem \ref{te2}. From the proof we get that the dual minimum distance must be
$q-1=4$ and points $P_1$, $P_2$, $P_3$, $P_4$ of the support of a minimum-weight
codeword must lie on the line through $P_0$ and $P_\infty$ (whose equation is
$x=0$). If $\F_{125}=\langle \beta \rangle$ (with $\beta^2+4\beta+2=0$) and we
compute the support of a (random) minimum-weight codeword of $C(4,1,1)^\perp$, we
find out the points of affine coordinates
\begin{equation*}
 P_1:=(0,\beta^3), \ \ P_2:=(0,\beta^9), \ \ P_3:=(0,\beta^{15}), \ \
P_4:=(0,\beta^{21}).
\end{equation*}
As one can easily see, they all lie on the line of equation $x=0$. 
\end{example}

\section{Conclusions}
In this paper we geometrically describe the supports of the minimum-weight codewords of the duals of two-point codes from the Hermitian curve. We use a geometric approach, characterizing the
dual minimum distance in terms of non-vanishing conditions on cohomology groups and zero-dimensional schemes in the plane.

\section*{Acknowledgement}
The authors would like to thank the referees for suggestions that improved the presentation of this work.

\end{document}